\documentclass[11pt]{amsart}

\usepackage{amssymb}
\usepackage[mathscr]{euscript}
\usepackage{calrsfs}
\usepackage{graphicx}
\usepackage{pb-diagram,pb-xy} 
\usepackage[all,cmtip,arrow]{xy} 
\usepackage[margin=1in]{geometry}  
\usepackage{hyperref} 
\usepackage{xcolor}

\numberwithin{equation}{section}

\theoremstyle{plain}
\newtheorem{theorem}[equation]{Theorem}
\newtheorem{thm}[equation]{Theorem}
\newtheorem{lemma}[equation]{Lemma}

\newtheorem{corollary}[equation]{Corollary}

\newtheorem*{proposition*}{Proposition}

\newtheorem{prop}[equation]{Proposition}

\newtheorem*{claim*}{Claim}

\theoremstyle{definition}
\newtheorem{definition}[equation]{Definition}

\newtheorem{remark}[equation]{Remark}

\newcommand{\isom}{\cong}                       

\newcommand{\Wdge}{\bigvee}                     

\newcommand{\cat}[1]{\mathscr{#1}}              
\newcommand{\twocat}[1]{\underline{\mathrm{#1}}} 

\DeclareMathOperator*{\colim}{colim}

\newcommand{\Hom}{\operatorname{Hom} }

\newcommand{\Set}{\mathsf{Set}}

\newcommand{\sSet}{\mathsf{sSet}}

\newcommand{\cof}{\rightarrowtail}              

\begin{document}

\title{Coalgebraic models for combinatorial model categories}
\author{Michael Ching and Emily Riehl}

\begin{abstract}
We show that the category of algebraically cofibrant objects in a combinatorial and simplicial model category $\cat{A}$ has a model structure that is left-induced from that on $\cat{A}$. In particular it follows that any presentable model category is Quillen equivalent (via a single Quillen equivalence) to one in which all objects are cofibrant.
\end{abstract}

\maketitle

Let $\cat{A}$ be a cofibrantly generated model category. Garner's version of the small object argument, described in \cite{garner:2012}, shows that there is a comonad $c$ on $\cat{A}$ for which the underlying endofunctor is a cofibrant replacement functor for the model structure on $\cat{A}$. If $X$ has a coalgebra structure for the comonad $c$, then $X$ is a retract of $cX$, so in particular $X$ is cofibrant. We therefore refer to a coalgebra for the comonad $c$ as an \emph{algebraically cofibrant} object of $\cat{A}$, though this terminology hides the fact that we have chosen a specific comonad $c$ and that there is potentially more than one coalgebra structure on a given cofibrant object.

When $c$ is defined via the Garner small object argument, the $c$-coalgebras include the `presented cell complexes' defined as chosen composites of pushouts of coproducts of generating cofibrations, with morphisms between them given by maps that preserve the cellular structure (see \cite{athorne:2012}).

We fix such a comonad $c$ and denote the category of $c$-coalgebras by $\cat{A}_c$. The forgetful functor $u: \cat{A}_c \to \cat{A}$ has a right adjoint given by taking the cofree coalgebra on an object. We abuse notation slightly and denote this right adjoint also by $c: \cat{A} \to \cat{A}_c$.

Our goal in this paper is a study of the forgetful/cofree adjunction
\begin{equation} \label{eq:adj} u : \cat{A}_c \rightleftarrows \cat{A} : c. \end{equation}
We show that when $\cat{A}$ is a \emph{combinatorial} and \emph{simplicial} model category, there is a (combinatorial and simplicial) model structure on $\cat{A}_c$ that is `left-induced' by that on $\cat{A}$. This means that a morphism $g$ in $\cat{A}_c$ is a weak equivalence or cofibration if and only if $u(g)$ is a weak equivalence, or respectively a cofibration, in $\cat{A}$. Given this, it is easily follows that (\ref{eq:adj}) is a Quillen equivalence.

We conjecture that the result holds without the hypothesis that $\cat{A}$ be simplicial but we have not been able to prove it in this generality.

Dugger shows in \cite{dugger:2001} that for any \emph{presentable} model category $\cat{C}$ (that is, a model category that is Quillen equivalent to a combinatorial model category), there exists a Quillen equivalence
\[ \cat{A} \rightleftarrows \cat{C} \]
in which $\cat{A}$ is combinatorial and simplicial. Composing with (\ref{eq:adj}) we obtain a single Quillen equivalence
\[ \cat{A}_c \rightleftarrows \cat{C} \]
such that every object in $\cat{A}_c$ is cofibrant. This improves on Dugger's~\cite[Corollary 1.2]{dugger:2001} in that we have a single Quillen equivalence rather than a zigzag.

Our result is dual to a corresponding theorem of Nikolaus~\cite{nikolaus:2011} which shows that the category of algebraically fibrant objects has a right-induced model structure, provided that every trivial cofibration is a monomorphism. Our approach is rather different and requires instead the conditions that the underlying model structure be combinatorial (not just cofibrantly generated) and simplicially-enriched.

Our result is an application of work of Bayeh et al.~\cite{bayeh/hess/karpova/kedziorek/riehl/shipley:2014} in which general conditions are established for the existence of left-induced model structures. There are two main parts: (1) we show that when $\cat{A}$ is locally presentable, the category $\cat{A}_c$ is also locally presentable; (2) we verify the conditions of \cite[Theorem 2.23]{bayeh/hess/karpova/kedziorek/riehl/shipley:2014} by showing that any morphism in $\cat{A}_c$ with the right lifting property with respect to all cofibrations is a weak equivalence.

The proof of part (1) depends crucially on the fact that accessible categories and accessible functors are closed under a certain class of 2-categorical limit constructions. This is closely related to the essential ingredient in Theorem~2.23 of \cite{bayeh/hess/karpova/kedziorek/riehl/shipley:2014}, which is a result of Makkai and Rosicky~\cite{makkai/rosicky:2013} that states that locally presentable categories with a cofibrantly generated weak factorization system, and appropriate functors between such, are closed under the same class of 2-categorical limits. We defer this part of our proof to the appendix, allowing for a more leisurely treatment.

\subsection*{Acknowledgments} This paper was written while both authors were visitors at the Mathematical Sciences Research Institute. The first author was supported by National Science Foundation grant DMS-1144149, and the second author by National Science Foundation Postdoctoral Research Fellowship DMS-1103790. The impetus for this paper was its application to the homotopic descent theory developed by the first author and Greg Arone, whom we thank for suggesting we look at it further. We also owe gratitude to Kathryn Hess and Brooke Shipley for their work on model categories of coalgebras over a comonad.

\section{Combinatorial simplicial model categories}

In this section, we explore the interaction between the locally presentable and simplicially enriched structures on a model category that is both simplicial and combinatorial. We also prove that there exists a cofibrant replacement comonad satisfying the conditions needed for our main theorem. We adopt the convention that all categories appearing in this paper are locally small,  meaning that there is a mere set of morphisms between any two fixed objects. First, recall the following definitions.

\begin{definition}
Let $\cat{A}$ be a cocomplete category. For a regular cardinal $\lambda$, an object $X$ in $\cat{A}$ is \emph{$\lambda$-presentable} if the functor $\cat{A}(X,-) : \cat{A} \to \Set$ preserves colimits of $\lambda$-filtered diagrams\footnote{A category $\cat{C}$ is \emph{$\lambda$-filtered} if every subcategory of $\cat{C}$ with fewer than $\lambda$ morphisms has a cocone in $\cat{C}$. A diagram is \emph{$\lambda$-filtered} when its indexing category is $\lambda$-filtered.}. A category $\cat{A}$ is \emph{locally $\lambda$-presentable} if it is cocomplete and has a (small) set $\cat{P}$ of $\lambda$-presentable objects  such that every object in $\cat{A}$ is isomorphic to a $\lambda$-filtered colimit of objects in $\cat{P}$.

A category is \emph{locally presentable} if it is locally $\lambda$-presentable for some regular cardinal $\lambda$. Locally presentable categories are abundant, particularly in `algebraic' or `simplicial' contexts. We refer the reader to \cite{adamek/rosicky:1994} for more details.
\end{definition}

Jeff Smith introduced the notion of a \emph{combinatorial} model category: one that is cofibrantly generated and for which the underlying category is locally presentable. In this paper we are interested in model categories $\cat{A}$ that are both combinatorial and \emph{simplicial}. The following result gives us control over how these two structures interact.

\begin{lemma} \label{lem:enriched-hom}
Let $\cat{A}$ be a combinatorial and simplicial model category. Then there are arbitrarily large cardinals $\lambda$ such that:
\begin{enumerate}
  \item $\cat{A}$ is locally $\lambda$-presentable;
  \item $\cat{A}$ is cofibrantly generated with a set of generating cofibrations for which the domains and codomains are $\lambda$-presentable objects;
  \item an object $X \in \cat{A}$ is $\lambda$-presentable if and only if the functor $\Hom(X,-): \cat{A} \to \sSet$, given by the simplicial enrichment of $\cat{A}$, preserves $\lambda$-filtered colimits.
\end{enumerate}
\end{lemma}
\begin{proof}
Since $\cat{A}$ is combinatorial, it is locally $\mu$-presentable for some cardinal $\mu$. There is also then some cardinal $\mu' > \mu$ such that the domains and codomains of the generating cofibrations are $\mu'$-presentable.

Now consider the functors $- \otimes \Delta^n : \cat{A} \to \cat{A}$, for $n \geq 0$, given by the simplicial tensoring on $\cat{A}$. These functors preserve all colimits and so by \cite[2.19]{adamek/rosicky:1994} there are arbitrarily large cardinals $\lambda$ such that all of these functors take $\lambda$-presentable objects to $\lambda$-presentable objects. Note that for any such $\lambda > \mu'$, conditions (1) and (2) of the lemma are satisfied.

Suppose that $X \in \cat{A}$ is $\lambda$-presentable and $(Y^j)_{j \in \cat{J}}$ is a $\lambda$-filtered diagram in $\cat{A}$. By construction of $\lambda$, $X \otimes \Delta^n$ is also $\lambda$-presentable and so
\begin{align*} \Hom(X, \colim_{\cat{J}} Y^j)_n &\cong \sSet(\Delta^n, \Hom(X,\colim_{\cat{J}} Y^j)) \cong \cat{A}( X \otimes \Delta^n, \colim_{\cat{J}} Y^j) \\ &\cong \colim_{\cat{J}} \cat{A}(X \otimes \Delta^n, Y^j)\cong \colim_{\cat{J}} \sSet(\Delta^n, \Hom(X, Y^j)) \cong \colim_{\cat{J}} \Hom(X,Y^j)_n \end{align*}
which implies that $\Hom(X,-)$ preserves $\lambda$-filtered colimits, because colimits in $\sSet$ are defined pointwise. Conversely, if $\Hom(X,-)$ preserves $\lambda$-filtered colimits, then by a similar sequence of bijections, so does $\cat{A}(X \otimes \Delta^n, -)$ for all $n \geq 0$. Hence, taking $n = 0$, $X$ is $\lambda$-presentable. Thus $\lambda$ satisfies condition (3) also.
\end{proof}

As described in the introduction, any cofibrantly generated model category $\cat{A}$ has a cofibrant replacement comonad $c: \cat{A} \to \cat{A}$ constructed using Garner's small object argument. Blumberg and Riehl \cite{blumberg/riehl:2012} have noted that when $\cat{A}$ is, in addition, a simplicial model category, the comonad $c$ can be taken to be simplicially enriched. In this case algebraically cofibrant objects also include `enriched cell complexes' built from tensors of generating cofibrations with arbitrary simplicial sets. For our main result we require a cofibrant replacement comonad that is both simplicial and preserves $\lambda$-filtered colimits.

\begin{lemma} \label{lem:simp-c}
Let $\cat{A}$ be a combinatorial and simplicial model category with $\lambda$ as in Lemma~\ref{lem:enriched-hom}. Then there is a simplicially-enriched cofibrant replacement comonad $c: \cat{A} \to \cat{A}$ that preserves $\lambda$-filtered colimits.
\end{lemma}
\begin{proof}
The functorial factorization produced by the simplicially enriched version of the Garner small object argument is defined by an iterated colimit process described in Theorem 13.2.1 of \cite{riehl:2014}. We apply this construction to the usual set of generating cofibrations to obtain a simplicially enriched functorial factorization of any map as a cofibration followed by a trivial fibration in the simplicial model structure on $\cat{A}$. The simplicially-enriched cofibrant replacement comonad $c$ is extracted by restricting the factorization to maps whose domain is initial; see \cite[Corollary 13.2.4]{riehl:2014}.

It remains to argue that $c$ preserves $\lambda$-filtered colimits. The functor $c$ is built from various colimit constructions, which commute with all colimits, and from functors $\mathrm{Sq}(i,-)\colon \cat{A}^2 \to \sSet$ defined for each generating cofibration $i$. For a morphism $f$ in $\cat{A}$, the simplicial set $\mathrm{Sq}(i,f)$ is the `space of commutative squares from $i$ to $f$' defined via the pullback
\[ \xymatrix{ \mathrm{Sq}(i,f) \ar[r] \ar[d] \ar@{}[dr]|(.2){\lrcorner} & \Hom(\mathrm{dom}\,i,\mathrm{dom}f) \ar[d] \\ \Hom(\mathrm{cod}\,i,\mathrm{cod}f) \ar[r] & \Hom(\mathrm{dom}\,i,\mathrm{cod}f)}\] The domain and codomain functors preserve all colimits, because they are defined pointwise. The enriched representables $\Hom(\mathrm{dom}\, i,-)$ and $\Hom(\mathrm{cod}\,i,-)$ preserve $\lambda$-filtered colimits by \ref{lem:enriched-hom}.(2). Finally, the pullback, a finite limit, commutes with all filtered colimits.
\end{proof}

The main result we need from this section is the following.

\begin{theorem} \label{thm:simp-c}
Let $\cat{A}$ be a combinatorial and simplicial model category and let $c$ be as in Lemma~\ref{lem:simp-c}. Then the category $\cat{A}_c$ of $c$-coalgebras is locally presentable, simplicially enriched, and the forgetful/cofree adjunction
\[ u : \cat{A}_c \rightleftarrows \cat{A} : c \]
is simplicial. Moreover, $\cat{A}_c$ is tensored and cotensored over simplicial sets.
\end{theorem}
\begin{proof}
We defer the proof that $\cat{A}_c$ is locally presentable to Proposition~\ref{prop:A_c-locally-presentable} in the appendix.

It is well-known that the comonadic adjunction for a simplicially enriched comonad on a simplicial category is simplicially enriched; see, e.g.,~\cite[\S3.2]{hess:2010}. The tensor $X \otimes K$ of a $c$-coalgebra $X$ with a simplicial set $K$ is defined as in \cite[Lemma 3.11]{hess:2010}. Observe that $u (X \otimes K) \cong uX \otimes K$  because left adjoints of simplicially enriched adjunctions preserve tensors \cite[3.7.10]{riehl:2014}. In this way, each simplicial set defines a functor $- \otimes K \colon \cat{A}_c \to \cat{A}_c$, which preserves all colimits because these colimits are created by the forgetful functor $u: \cat{A}_c \to \cat{A}$, and tensoring with $K$ preserves colimits in $\cat{A}$. By the special adjoint functor theorem, any cocontinuous functor between locally presentable categories has a right adjoint,\footnote{See \protect\cite[1.58]{adamek/rosicky:1994} for a proof that locally presentable categories are co-wellpowered.} so $\cat{A}_c$ is also cotensored over simplicial sets.
\end{proof}

It turns out that the $\lambda$-presentable objects in $\cat{A}_c$ can be easily described.

\begin{lemma} \label{lem:pres-A_c}
Let $\lambda$ and $c$ be as in Lemma~\ref{lem:simp-c}. Then a $c$-coalgebra $X \in \cat{A}_c$ is $\lambda$-presentable if and only if the underlying object $uX$ is $\lambda$-presentable in $\cat{A}$.
\end{lemma}
\begin{proof}
First note that because the comonad $c$ preserves $\lambda$-filtered colimits and the left adjoint $u \colon \cat{A}_c \to \cat{A}$ creates them, the cofree functor $c \colon \cat{A} \to \cat{A}_c$ also preserves $\lambda$-filtered colimits.

Now suppose that $uX$ is $\lambda$-presentable and consider a $\lambda$-filtered diagram $(Y^j)_{j \in \cat{J}}$ in $\cat{A}_c$. Then we have
\[ \cat{A}_c(X,\colim Y^j) = \lim (\cat{A}(uX,u\colim Y^j) \rightrightarrows \cat{A}(uX,ucu \colim Y^j)). \]
Since $c$, $u$, $\cat{A}(uX,-)$, and finite limits of sets commute with $\lambda$-filtered colimits, so too does $\cat{A}_c(X,-)$ and thus $X$ is $\lambda$-presentable.

Conversely, suppose that $X$ is $\lambda$-presentable and consider a $\lambda$-filtered diagram $(A^j)_{j\in \cat{J}}$ in $\cat{A}$. Then, by the $u \dashv c$ adjunction, we have
\[ \cat{A}(uX,\colim A^j) \isom \cat{A}_c(X,c \colim A^j). \]
Since $c$ and $\cat{A}_c(X,-)$ commute with $\lambda$-filtered colimits, so too does $\cat{A}(uX,-)$ and so $uX$ is $\lambda$-presentable.
\end{proof}

We also need the following result from \cite[3.6]{raptis/rosicky:2014} which slightly strengthens that of Dugger \cite[7.3]{dugger:2001}.

\begin{lemma}[Raptis-Rosick\'{y}] \label{lem:raptis}
Let $\cat{A}$ be a cofibrantly generated model category with $\lambda$ as in Lemma~\ref{lem:enriched-hom}.(2). Then taking $\lambda$-filtered colimits preserves weak equivalences.
\end{lemma}
\begin{proof}
We wish to show that the colimit functor $\colim\colon \cat{A}^\cat{J} \to \cat{A}$ preserves weak equivalences, when $\cat{J}$ is $\lambda$-filtered. Endowing the domain with the projective model structure, $\colim$ is left Quillen. A given pointwise weak equivalence between $\cat{J}$-indexed diagrams may be factored as a projective trivial cofibration followed by a pointwise trivial fibration. As $\colim$ sends projective trivial cofibrations to trivial cofibrations, it suffices to show that $\colim$ carries pointwise trivial fibrations to trivial fibrations. This follows because the domains and codomains of the generating cofibrations are $\lambda$-presentable and $\cat{J}$ is $\lambda$-filtered.
\end{proof}

\section{The model structure on algebraically cofibrant objects}

We now turn to the proof of our main theorem, that for a combinatorial simplicial model category $\cat{A}$, the adjunction $u \dashv c$ `left-induces' a model structure on $\cat{A}_c$. We therefore make the following definitions.

\begin{definition} \label{def:A_c-weq}
A morphism $f$ in $\cat{A}_c$ is a \emph{weak equivalence} if the underlying morphism $u(f)$ is a weak equivalence in $\cat{A}$, and $f$ is a \emph{cofibration} if $u(f)$ is a cofibration in $\cat{A}$.
\end{definition}

According to \cite[Theorem 2.23]{bayeh/hess/karpova/kedziorek/riehl/shipley:2014} it is now sufficient to show that every morphism $g$ in $\cat{A}_c$, which has the right lifting property with respect to all cofibrations, is a weak equivalence. Our basic strategy is to show that weak equivalences are detected by the simplicial enrichment. For this we use the following bar resolutions of the objects in $\cat{A}_c$.

\begin{definition} \label{def:bar}
By Theorem~\ref{thm:simp-c}, $\cat{A}_c$ is locally $\lambda$-presentable for some regular cardinal $\lambda$, and, by increasing $\lambda$ if necessary, we may assume that $\lambda$ also satisfies the conditions of Lemma~\ref{lem:enriched-hom} (and hence also Lemmas~\ref{lem:pres-A_c} and \ref{lem:raptis}). Choose a (small) set $\cat{P}$ of $\lambda$-presentable objects in $\cat{A}_c$ such that every object of $\cat{A}_c$ is a $\lambda$-filtered colimit of objects of $\cat{P}$.

Let $\Hom_c(-,-)$ denote the simplicial enrichment of $\cat{A}_c$. For any $X \in \cat{A}_c$ we then define a simplicial object $\mathfrak{B}_\bullet(X)$ in $\cat{A}_c$ with
\[ \mathfrak{B}_r(X) := \Wdge_{P_0,\dots,P_r \in \cat{P}} P_0 \otimes [\Hom_c(P_0,P_1) \times \dots \times \Hom_c(P_r,X)] \]
where the face maps are given by composition in the simplicial category $\cat{A}_c$ and degeneracy maps are given by the unit maps $* \to \Hom_c(P_i,P_i)$. This simplicial object is naturally augmented by a map
\[ \mathfrak{B}_0(X) \to X \]
which therefore induces a natural transformation with components
\[ \eta_X: |\mathfrak{B}_\bullet(X)| \to X. \]
\end{definition}

Notice that the enriched bar construction $|\mathfrak{B}_\bullet(X)|$ plays the role of the homotopy colimit of the canonical diagram on the overcategory $\cat{P} \downarrow X$. The next result is then the analogue of Dugger's \cite[4.7]{dugger:2001}.

\begin{lemma} \label{lem:bar}
For each $X \in \cat{A}_c$, the map $\eta_X$ is a weak equivalence in $\cat{A}_c$.
\end{lemma}
\begin{proof}
For $P \in \cat{P}$, the augmented simplicial object $\mathfrak{B}_\bullet(P) \to P$ has extra degeneracies given by the unit map $* \to \Hom_c(P,P)$ and so the map $\eta_P$ is a simplicial homotopy equivalence in $\cat{A}_c$. Therefore $u(\eta_P)$ is a simplicial homotopy equivalence, and hence a weak equivalence, in $\cat{A}$. So $\eta_P$ is a weak equivalence in $\cat{A}_c$.

For any $X \in \cat{A}_c$, we can write
\[ X \isom \colim_{j \in \cat{J}} P^j \]
for some $\lambda$-filtered diagram $(P^j)_{j \in \cat{J}}$ in $\cat{A}_c$ where $P^j \in \cat{P}$. We now claim that for any $P \in \cat{P}$, the natural map of simplicial sets
\begin{equation} \label{eq:hom-colim} \colim_{\cat{J}} \Hom_c(P,P^j) \to \Hom_c(P,\colim_{\cat{J}}P^j) \isom \Hom_c(P,X) \end{equation}
is an isomorphism. Since colimits of simplicial sets are formed levelwise, it is sufficient to show that, for each integer $n \geq 0$, the map
\[ \colim_{\cat{J}} \cat{A}_c(P \otimes \Delta^n,P^j) \to \cat{A}_c(P \otimes \Delta^n,\colim_{\cat{J}}P^j) \]
is an isomorphism of sets, i.e. that $P \otimes \Delta^n$ is a $\lambda$-presentable object of $\cat{A}_c$.

Now $u(P)$ is $\lambda$-presentable in $\cat{A}$ by Lemma~\ref{lem:pres-A_c} and so $u(P \otimes \Delta^n) = u(P) \otimes \Delta^n$ is $\lambda$-presentable by condition (3) of Lemma~\ref{lem:enriched-hom}. So by Lemma~\ref{lem:pres-A_c} again, $P \otimes \Delta^n$ is $\lambda$-presentable in $\cat{A}_c$. This establishes the isomorphisms (\ref{eq:hom-colim}).

Since finite products, tensors, coproducts, and geometric realization all commute with $\lambda$-filtered colimits, it follows that the natural map
\[ \colim_{\cat{J}}|\mathfrak{B}_\bullet(P^j)| \to |\mathfrak{B}_\bullet(X)| \]
is an isomorphism. Thus the map $u(\eta_X)$ can be written as the $\lambda$-filtered colimit of the maps $u(\eta_{P^j})$, each of which is a weak equivalence in $\cat{A}$. By Lemma~\ref{lem:raptis}, $u(\eta_X)$ is a weak equivalence in $\cat{A}$, and hence $\eta_X$ is a weak equivalence in $\cat{A}_c$.
\end{proof}

We now use the above bar resolutions to complete the proof of our main theorem.

\begin{theorem} \label{thm:main}
Let $\cat{A}$ be a combinatorial and simplicial model category and $c$ as in Lemma~\ref{lem:simp-c}. Then the weak equivalences and cofibrations of Definition~\ref{def:A_c-weq} make $\cat{A}_c$ into a combinatorial and simplicial model category such that the adjunction $u \dashv c$ is a simplicial Quillen equivalence.
\end{theorem}
\begin{proof}
To get the model structure on $\cat{A}_c$ we apply \cite[2.23]{bayeh/hess/karpova/kedziorek/riehl/shipley:2014}. We have already shown in Theorem~\ref{thm:simp-c} that $u \dashv c$ is an adjunction between locally presentable categories.

Now let $g: X \to Y$ be a morphism in $\cat{A}_c$ that has the right lifting property with respect to all cofibrations. We claim first that for any $P \in \cat{A}_c$, $g$ induces a weak equivalence of simplicial sets
\[ g_*: \Hom_c(P,X) \to \Hom_c(P,Y). \]
In fact, we show that $g_*$ is a trivial fibration of simplicial sets. To see this, consider a lifting problem
\[ \begin{diagram}
     \node{K} \arrow{e} \arrow{s,V} \node{\Hom_c(P,X)} \arrow{s,r}{g_*} \\
     \node{L} \arrow{e} \arrow{ne,..} \node{\Hom_c(P,Y)}
\end{diagram} \]
where $K \cof L$ is a cofibration of simplicial sets. By adjointness, a lift here corresponds to a lift of the diagram
\[ \begin{diagram}
     \node{P \otimes K} \arrow{e} \arrow{s} \node{X} \arrow{s,r}{g} \\
     \node{P \otimes L} \arrow{e} \arrow{ne,..} \node{Y}
\end{diagram} \]
so it is sufficient to show that $P \otimes K \to P \otimes L$ is a cofibration in $\cat{A}_c$. By definition, this is the case if the induced map $u(P \otimes K) \to u(P \otimes L)$ is a cofibration in $\cat{A}$. But tensors are preserved by the left adjoint $u$ and the map
\[ u(P) \otimes K \to u(P) \otimes L \]
is a cofibration in the simplicial model category $\cat{A}$ because $u(P)$ is cofibrant for any $P \in \cat{A}_c$. This establishes that $g_*$ is an trivial fibration, and hence a weak equivalence, of simplicial sets.

Now consider the following diagram in $\cat{A}_c$ where $\mathfrak{B}_\bullet(-)$ is as in Definition~\ref{def:bar}:
\begin{equation} \label{eq:diag-bar} \begin{diagram}
  \node{|\mathfrak{B}_\bullet(X)|} \arrow{e,t}{\eta_X} \arrow{s,l}{|\mathfrak{B}_\bullet(g)|} \node{X} \arrow{s,r}{g} \\
  \node{|\mathfrak{B}_\bullet(Y)|} \arrow{e,t}{\eta_Y} \node{Y}
\end{diagram} \end{equation}
We know that the maps $\eta_X$ and $\eta_Y$ are weak equivalences by Lemma~\ref{lem:bar}. We claim that the map labelled $|\mathfrak{B}_\bullet(g)|$ is also a weak equivalence.

For any objects $P_0,\dots,P_r \in \cat{P}$, the map $g$ induces weak equivalences of simplicial sets
\[ \Hom_c(P_0,P_1) \times \dots \times \Hom_c(P_r,X) \to \Hom_c(P_0,P_1) \times \dots \times \Hom_c(P_r,Y) \]
by our previous calculation. Therefore, since $u(P_0)$ is cofibrant and $\cat{A}$ is a simplicial model category, the induced maps
\[ u(P_0) \otimes [\Hom_c(P_0,P_1) \times \dots \times \Hom_c(P_r,X)] \to u(P_0) \otimes [\Hom_c(P_0,P_1) \times \dots \times \Hom_c(P_r,Y)] \]
are weak equivalences between cofibrant objects in $\cat{A}$. Taking coproducts preserves weak equivalences between cofibrant objects, and $u$ commutes with tensors and coproducts, so $g$ induces weak equivalences
\[ u(\mathfrak{B}_r(X)) \to u(\mathfrak{B}_r(Y)) \]
for each $r$. Since the simplicial objects $u(\mathfrak{B}_\bullet(-))$ are Reedy cofibrant, it follows that the induced map on geometric realizations is a weak equivalence too. This implies $|\mathfrak{B}_\bullet(g)|$ is a weak equivalence in $\cat{A}_c$.

Finally, we note that the Definition~\ref{def:A_c-weq} implies that weak equivalences in $\cat{A}_c$ have the 2-out-of-3 property, and so we deduce from (\ref{eq:diag-bar}) that $g$ is a weak equivalence.

This now establishes, using \cite[2.23]{bayeh/hess/karpova/kedziorek/riehl/shipley:2014}, that $\cat{A}_c$ has the desired model structure and it follows immediately that $u \dashv c$ is a simplicial Quillen adjunction; see \cite[Lemma 2.23]{bayeh/hess/karpova/kedziorek/riehl/shipley:2014}. But a morphism $g: X \to c(A)$ in $\cat{A}_c$ is a weak equivalence if and only if the map $u(g): u(X) \to uc(A)$ is a weak equivalence in $\cat{A}$. Since the counit map $uc(A) \to A$ is always a weak equivalence ($uc$ is our original cofibrant replacement functor) it follows that $u(g)$ is a weak equivalence if and only if the adjoint map $g^{\#} : u(X) \to A$ is a weak equivalence. Hence $u \dashv c$ is a Quillen equivalence.
\end{proof}

\section{Coalgebraic models for presentable model categories}

\begin{definition}
We say that a model category is \emph{presentable} if it is Quillen equivalent (via a zigzag of Quillen equivalences) to some combinatorial model category. Note that very many familiar model categories are presentable. For example, if one accepts a certain set-theoretic axiom about the existence of large cardinals (Vop\v{e}nka's Principle), then every cofibrantly generated model category is presentable by \cite{raptis:2009}.
\end{definition}

The following theorem follows from \cite[1.1]{dugger:2001} using \cite[6.5]{dugger:2001b}.

\begin{theorem}[Dugger] \label{thm:dugger}
For every presentable model category $\cat{M}$ there is a Quillen equivalence:
\[ l : \cat{A} \rightleftarrows \cat{M} : r \qquad l \dashv r \]
where $\cat{A}$ is a combinatorial and simplicial model category. Moreover, if $\cat{M}$ is also a simplicial model category then the adjunction $l \dashv r$ may also be chosen to be simplicial.
\end{theorem}
\begin{proof}
Dugger does not explicitly state the claim about what happens when $\cat{M}$ is simplicial but it follows from his proof. The key step in his argument is \cite[5.10]{dugger:2001b}, which says that given Quillen equivalences (with left adjoints on top)
\[ U\mathcal{C}/S \rightleftarrows \cat{N} \quad \text{and} \quad \cat{M} \rightleftarrows \cat{N}, \]
where $U\mathcal{C}/S$ is a left Bousfield localization of the projective model structure on the category $U\mathcal{C}$ of simplicial presheaves on a small category $\mathcal{C}$, there is a `lifted' single Quillen equivalence
\[ U\mathcal{C}/S \rightleftarrows \cat{M}. \]
We claim that when $\cat{M}$ is a simplicial model category, this last adjunction can be taken to be simplicial.

Dugger constructs the lifted Quillen equivalence using \cite[2.3]{dugger:2001b} which says that any functor $\gamma: \mathcal{C} \to \cat{M}$ can be extended (up to weak equivalence) to the left functor in a Quillen adjunction
\[ \operatorname{Re}: U\mathcal{C} \rightleftarrows \cat{M} : \operatorname{Sing}. \]
We need to show that when $\cat{M}$ is a simplicial model category this adjunction can be chosen to be simplicial. To see that this is the case, we choose an objectwise cofibrant replacement $\gamma^{\mathsf{cof}}$ for $\gamma$, and define $\operatorname{Re}$ on a presheaf $F: \mathcal{C}^{op} \to \sSet$ by the coend
\[ \operatorname{Re}(F) = F \otimes_{\mathcal{C}} \gamma^{\mathsf{cof}} := \operatorname{coeq} \left( \coprod_{a \to b \in \mathcal{C}} F(b) \otimes \gamma^{\mathsf{cof}}(a) \rightrightarrows \coprod_{c \in \mathcal{C}} F(c) \otimes \gamma^{\mathsf{cof}}(c) \right) \]
where $\otimes$ denotes the tensoring of $\cat{M}$ by simplicial sets, and $\operatorname{Sing}$ on $M \in \cat{M}$ by
\[ \operatorname{Sing}(M)(c) := \Hom_{\cat{M}}(\gamma^{\mathsf{cof}}(c),M) \]
where $\Hom_{\cat{M}}(-,-)$ denotes the simplicial enrichment of $\cat{M}$. The right adjoint $\operatorname{Sing}$ preserves fibrations and trivial fibrations, so $\operatorname{Re} \dashv \operatorname{Sing}$ is a Quillen adjunction, and the value of $\operatorname{Re}$ on the representable (discrete) presheaf $\mathcal{C}(-,c)$ is $\gamma^{\mathsf{cof}}(c)$. Since tensoring with simplicial sets is objectwise in $U\mathcal{C}$, it is easy to check that $\operatorname{Re} \dashv \operatorname{Sing}$ is also a simplicial adjunction.
\end{proof}

\begin{remark}
It follows from Dugger's approach that the category $\cat{A}$ in Theorem~\ref{thm:dugger} can be taken to be of the form $U\mathcal{C}/S$ (a left Bousfield localization of simplicial presheaves on a small category $\mathcal{C}$). It is not hard to see that, for such $\cat{A}$, any infinite cardinal $\lambda$ satisfies the conditions of Lemma~\ref{lem:enriched-hom}.
\end{remark}

\begin{corollary}
For any presentable model category $\cat{C}$ there is a Quillen equivalence:
\[ l : \cat{A}_c \rightleftarrows \cat{C} : r \qquad l \dashv r \]
where $\cat{A}_c$ is a combinatorial and simplicial model category in which every object is cofibrant. Moreover, if $\cat{C}$ is also a simplicial model category then the adjunction $l \dashv r$ may be chosen to be simplicial.
\end{corollary}
\begin{proof}
Compose the Quillen equivalences of Theorems~\ref{thm:main} and \ref{thm:dugger}.
\end{proof}

\appendix
\section{Local presentability for coalgebras over accessible comonads}

We say that an endofunctor $c\colon \cat{A} \to \cat{A}$ on a locally presentable category is \emph{accessible} if there is some regular cardinal $\lambda$ such that $\cat{A}$ is locally $\lambda$-presentable and $c$ preserves $\lambda$-filtered colimits. In this section, we prove the following result:

\begin{prop} \label{prop:A_c-locally-presentable}
Let $\cat{A}$ be a locally presentable category and $c: \cat{A} \to \cat{A}$ an accessible comonad. Then the category $\cat{A}_c$ of $c$-coalgebras is locally presentable.
\end{prop}

Note, however, that we do not have control over the `index of accessibility' of $\cat{A}_c$. In particular, we do not claim that $\cat{A}_c$ is locally $\lambda$-presentable for the same cardinal $\lambda$ as above.

Proposition~\ref{prop:A_c-locally-presentable} is dual to \cite[2.78]{adamek/rosicky:1994} which says that the category of algebras over an accessible monad is locally presentable. Our proof is basically the same but we give a detailed outline of the proof for the reader who is not overly familiar with these notions.

First note that, by the following result, it is sufficient to show that the category $\cat{A}_c$ is \emph{accessible}. (This means that there is some regular cardinal $\mu$ such that $\cat{A}_c$ has $\mu$-filtered colimits and such that there is a set of $\mu$-presentable objects in $\cat{A}_c$ that generates the category under $\mu$-filtered colimits.)

\begin{thm}[{Ad\'{a}mek-Rosick\'{y} \cite[2.47]{adamek/rosicky:1994}}]\label{thm:TFAE} The following are equivalent for a category $\cat{C}$:
\begin{itemize}
 \item $\cat{C}$ is accessible and complete,
 \item $\cat{C}$ is accessible and cocomplete,
 \item $\cat{C}$ is locally presentable.
 \end{itemize}
\end{thm}

Any category of coalgebras inherits the colimits present in its underlying category. In particular, $\cat{A}_c$ is cocomplete, so Theorem \ref{thm:TFAE} implies that if $\cat{A}_c$ is accessible then it is locally presentable (and, in particular, complete).

The proof that $\cat{A}_c$ is accessible makes use of a class of 2-categorical limits called \emph{PIE-limits}, which we describe in Definition~\ref{def:PIE} below. By a theorem of Makkai and Par\'{e}, the 2-category $\twocat{ACC}$ of accessible categories, accessible functors, and natural transformations has all PIE-limits, created by the forgetful functor $\twocat{ACC} \to \twocat{CAT}$ \cite[5.1.6]{makkai/pare:1989}. This tells us that any category constructed from a diagram of accessible categories and accessible functors, via PIE-limits, is accessible. In particular, we prove Proposition~\ref{prop:A_c-locally-presentable} by showing that the category $\cat{A}_c$ can be constructed from $\cat{A}$ and $c$ by a sequence of PIE-limits.

A diagram consisting of categories, functors, and natural transformations can be encoded by a 2-functor into the 2-category $\twocat{CAT}$. A \emph{2-limit} of such a 2-functor is a category equipped with a universal `cone' over the diagram, where the cone can have a variety of shapes involving functors and natural transformations. The particular 2-limits that we use here are elementary, and we can give an explicit construction.

 \begin{definition}\label{def:PIE}
 A \emph{PIE-limit} is a 2-limit that can be constructed as a composite of the following 2-limits:
 \begin{itemize}
  \item The \emph{product} of two categories $\cat{X}$ and $\cat{Y}$ is the usual cartesian product. Note that it satisfies an enriched universal property: for any category $\cat{A}$ the map induced by the projections defines an isomorphism of functor categories \[ \twocat{CAT}(\cat{A},\cat{X} \times \cat{Y}) \cong \twocat{CAT}(\cat{A},\cat{X}) \times \twocat{CAT}(\cat{A},\cat{Y}).\]
  \item The \emph{inserter} for a parallel pair of functors $f,g \colon \cat{X} \to \cat{Y}$ is a category $\cat{I}$ together with a functor $i \colon \cat{I} \to \cat{X}$ and a natural transformation
  \[ \xymatrix{ \cat{I} \ar@/^3ex/[r]^{fi} \ar@/_3ex/[r]_{gi} \ar@{}[r]|{\Downarrow \alpha} & \cat{Y}} \] that is universal with this property. Explicitly, the inserter is the category containing an object for each $x \in \cat{X}$ and morphism $fx \to gx$ in $\cat{Y}$. Morphisms are maps $x \to x'$ in $\cat{X}$ that yield commutative `naturality' squares in $\cat{Y}$.
  \item  The \emph{equifier} of a parallel pair of natural transformations
  \[ \xymatrix{ \cat{X} \ar@/^3ex/[r]^f \ar@/_3ex/[r]_g \ar@{}[r]|{\alpha\Downarrow \Downarrow \beta} & \cat{Y}}\]
  is a category $\cat{E}$ together with a functor $e \colon \cat{E} \to \cat{X}$ so that $\alpha e = \beta e$ and that is universal with this property. Explicitly, the equifier is the full subcategory $\cat{E} \subset \cat{X}$ spanned by objects $x \in \cat{X}$ for which $\alpha_x = \beta_x$.
 \end{itemize}
\end{definition}

For general 2-categories such as $\twocat{ACC}$, products, inserters, and equifiers are defined representably by products, inserters, and equifiers in $\twocat{CAT}$. More details can be found in the article \cite{kelly:1989}, which provides an excellent introduction to this topic.

\begin{proof}[Proof of Proposition \ref{prop:A_c-locally-presentable}]
We construct $\cat{A}_c$ as a PIE-limit in three stages. First, form the inserter $\cat{I}$ of the pair of accessible functors $1_\cat{A},c \colon \cat{A} \to \cat{A}$. Its objects are maps $a \to ca$ in $\cat{A}$ for some $a \in \cat{A}$. Morphisms are maps $a \to a'$ in $\cat{A}$ making the evident `naturality' squares commute. This category comes equipped with a forgetful functor $i \colon \cat{I} \to\cat{A}$ and a natural transformation
\[ \xymatrix{ \cat{I} \ar@/^3ex/[r]^{i} \ar@/_3ex/[r]_{ci} \ar@{}[r]|{\Downarrow \alpha} & \cat{A}} \] as in Definition \ref{def:PIE}. By \cite[5.1.6]{makkai/pare:1989}, $\cat{I}$ and $i$ are then accessible.

Next form the equifier $\cat{E}$ of the pair of natural transformations
\[ \xymatrix@=40pt{ \cat{I} \ar@/^3ex/[r]^i \ar@/_3ex/[r]_{i} \ar@{}[r]|{\mathrm{id}_i\Downarrow \Downarrow \epsilon i\cdot \alpha} & \cat{A}}\]
where $\epsilon$ is the counit for the comonad $c$. By Definition \ref{def:PIE}, $\cat{E}$ is the full subcategory of $\cat{I}$ whose objects are those $a \to ca$ of $\cat{I}$ such that the composite $a \to ca \xrightarrow{\epsilon_a} a$ is the identity at $a$. Write $e \colon \cat{E} \to \cat{I}$ for the inclusion functor. By \cite[5.1.6]{makkai/pare:1989} again, $\cat{E}$ and $e$ are accessible.

Finally, $\cat{A}_c$ is given by the equifier of the pair of natural transformations
\[ \xymatrix@=60pt{ \cat{E} \ar@/^3ex/[r]^{ie} \ar@/_3ex/[r]_{c^2ie} \ar@{}[r]|{c\alpha e \cdot \alpha e\Downarrow \Downarrow \delta ie \cdot  \alpha e} & \cat{A}}\]
where $\delta$ is the comultiplication for the comonad $c$. By Definition \ref{def:PIE}, this equifier is the full subcategory of $\cat{I}$ whose objects are maps $h \colon a \to ca$ in $\cat{A}$ so that the diagrams
\[ \xymatrix{ a \ar[r]^h \ar[dr]_{1_a} & ca \ar[d]^{\epsilon_a} & & a \ar[r]^h \ar[d]_h & ca \ar[d]^{\delta_a} \\ & a & & ca \ar[r]_{ch} & c^2a}\] commute, i.e. is equal to $\cat{A}_c$. Hence by \cite[5.1.6]{makkai/pare:1989} one more time, $\cat{A}_c$ is accessible, and so by Theorem~\ref{thm:TFAE}, $\cat{A}_c$ is locally presentable.
\end{proof}

\begin{remark}
Note the index of accessibility of $\cat{A}_c$ established by Proposition \ref{prop:A_c-locally-presentable} may be larger than the index of accessibility of both $\cat{A}$ and $c$. Thus the cardinal $\lambda$ that we use to prove Theorem \ref{thm:main} needs to be chosen to be large enough so that both Lemma \ref{lem:enriched-hom} and Theorem \ref{thm:simp-c} apply.
\end{remark}

\begin{remark}
Another type of PIE-limit is called a \emph{pseudopullback}. The (strict 2-)pullback of a diagram of categories and functors is a pseudo-pullback if one of the functors is an \emph{isofibration}, i.e., a functor with the right lifting property with respect to the inclusion of the terminal category into the category with two objects connected by an isomorphism.

The main theorem of Makkai and Rosick\'{y}~\cite{makkai/rosicky:2013} is that the 2-category of locally presentable categories equipped with a cofibrantly generated weak factorization system, colimit-preserving functors that preserve the left classes, and all natural transformations has PIE-limits created by the forgetful functor to $\twocat{CAT}$. The pseudopullback of a cofibrantly generated weak factorization system on $\cat{C}$ along a colimit-preserving functor $u \colon \cat{A} \to \cat{C}$ between locally presentable categories equipped with `trivial' weak factorization systems defines the left-induced weak factorization system on $\cat{A}$, which is therefore cofibrantly generated. This result enabled the proof of \cite[Theorem 2.23]{bayeh/hess/karpova/kedziorek/riehl/shipley:2014} which in turn was the key tool used in the current paper.
\end{remark}

\bibliographystyle{amsplain}
\bibliography{mcching}

\providecommand{\bysame}{\leavevmode\hbox to3em{\hrulefill}\thinspace}
\providecommand{\MR}{\relax\ifhmode\unskip\space\fi MR }
\providecommand{\MRhref}[2]{%
  \href{http://www.ams.org/mathscinet-getitem?mr=#1}{#2}
}
\providecommand{\href}[2]{#2}
\begin{thebibliography}{10}

\bibitem{adamek/rosicky:1994}
Ji{\v{r}}{\'{\i}} Ad{\'a}mek and Ji{\v{r}}{\'{\i}} Rosick{\'y}, \emph{Locally
  presentable and accessible categories}, London Mathematical Society Lecture
  Note Series, vol. 189, Cambridge University Press, Cambridge, 1994.
  \MR{1294136 (95j:18001)}

\bibitem{athorne:2012}
Thomas Athorne, \emph{The coalgebraic structure of cell complexes}, Theory
  Appl. Categ. \textbf{26} (2012), No. 11, 304--330. \MR{2948491}

\bibitem{bayeh/hess/karpova/kedziorek/riehl/shipley:2014}
M.~{Bayeh}, K.~{Hess}, V.~{Karpova}, M.~{Kedziorek}, E.~{Riehl}, and
  B.~{Shipley}, \emph{{Left-induced model structures and diagram categories}},
  ArXiv e-prints (2014).

\bibitem{blumberg/riehl:2012}
A.~J. {Blumberg} and E.~{Riehl}, \emph{{Homotopical resolutions associated to
  deformable adjunctions}}, ArXiv e-prints (2012).

\bibitem{dugger:2001}
Daniel Dugger, \emph{Combinatorial model categories have presentations}, Adv.
  Math. \textbf{164} (2001), no.~1, 177--201. \MR{1870516 (2002k:18022)}

\bibitem{dugger:2001b}
\bysame, \emph{Universal homotopy theories}, Adv. Math. \textbf{164} (2001),
  no.~1, 144--176. \MR{1870515 (2002k:18021)}

\bibitem{garner:2012}
Richard Garner, \emph{Understanding the small object argument}, Applied
  Categorical Structures \textbf{20} (2012), 103--141,
  10.1007/s10485-008-9126-7.

\bibitem{hess:2010}
Kathryn Hess, \emph{A general framework for homotopic descent and codescent},
  2010.

\bibitem{kelly:1989}
G.~M. Kelly, \emph{Elementary observations on {$2$}-categorical limits}, Bull.
  Austral. Math. Soc. \textbf{39} (1989), no.~2, 301--317. \MR{998024
  (90f:18004)}

\bibitem{makkai/rosicky:2013}
M.~{Makkai} and J.~{Rosick{\'y}}, \emph{{Cellular categories}}, ArXiv e-prints
  (2013).

\bibitem{makkai/pare:1989}
Michael Makkai and Robert Par{\'e}, \emph{Accessible categories: the
  foundations of categorical model theory}, Contemporary Mathematics, vol. 104,
  American Mathematical Society, Providence, RI, 1989. \MR{1031717 (91a:03118)}

\bibitem{nikolaus:2011}
Thomas Nikolaus, \emph{Algebraic models for higher categories}, Indag. Math.
  (N.S.) \textbf{21} (2011), no.~1-2, 52--75. \MR{2832482 (2012g:55029)}

\bibitem{raptis/rosicky:2014}
G.~{Raptis} and J.~{Rosick{\'y}}, \emph{{The accessibility rank of weak
  equivalences}}, ArXiv e-prints (2014).

\bibitem{raptis:2009}
George Raptis, \emph{On the cofibrant generation of model categories}, J.
  Homotopy Relat. Struct. \textbf{4} (2009), no.~1, 245--253. \MR{2520994
  (2010e:18015)}

\bibitem{riehl:2014}
Emily Riehl, \emph{Categorical homotopy theory}, Cambridge University Press,
  Cambridge, 2014, New Mathematical Monographs.

\end{thebibliography}

\end{document}